\UseRawInputEncoding
\documentclass[final,12pt]{elsarticle}
\usepackage[utf8]{inputenc}
\usepackage{ragged2e}
\justifying
\usepackage{bm}
\usepackage{geometry}
\usepackage{amsmath}
\usepackage{mathrsfs}
\usepackage{amsthm}
\usepackage{amssymb}
\newtheorem{theorem}{Theorem}[section]
\newtheorem{corollary}{Corollary}[section]
\newtheorem{lemma}{Lemma}[section]
\newtheorem{definition}{Definition}[section]

\newtheorem{proposition}{Proposition}[section]

\newenvironment{Proof of Theorem 1.4 (ii).}{\noindent{\textbf{Proof of Theorem 1.4 (ii).}}\ }{\hfill $\square$\par}

\numberwithin{equation}{section}

\usepackage{indentfirst}
\begin{document}
	\begin{frontmatter}  
		\title{Tur\'{a}n number for odd-ballooning of bipartite graphs\,\tnoteref{titlenote}}  
		\tnotetext[titlenote]{This work was supported by the National Natural Science Foundation  of China (Nos. 11871040, 12271337).}    %文章下标批注
		\author{Yanni Zhai}  
		%	\ead{xiyingyuan@shu.edu.cn}   
		\author{Xiying Yuan\corref{correspondingauthor}}
		\cortext[correspondingauthor]{Corresponding author. \\
			Email address: yannizhai2022@163.com (Yanni Zhai), xiyingyuan@shu.edu.cn (Xiying Yuan).}   
		\address{Department of Mathematics, Shanghai University, Shanghai 200444, P.R. China}  
		\begin{abstract}  
		Given a graph $H$ and an odd integer $t$ ($t\geq 3$),
		the odd-ballooning of $H$, denoted by $H(t)$, is the graph obtained from replacing each edge of $H$ by an odd cycle of length at least $t$ where the new vertices of the cycles are all distinct. In this paper, we determine the range of Tur\'{a}n numbers for odd-ballooning of  bipartite graphs when $t\geq 5$. As applications, we may deduce the Tur\'{a}n numbers for odd-ballooning of stars, paths and even cycles.
		\end{abstract}   
	\end{frontmatter}
	
	\section{Introduction}
	In this paper, we consider simple graphs without loops and multiedges.  The order of a graph $H=\left(V(H),\,E(H)\right)$ is the number of its vertices denoted by $\nu(H)$, and the size of a graph $H$ is the number of its edges denoted by $e(H)$. For a vertex $v\in V(H)$, the neighborhood of $v$ in $H$ is denoted by $N_{H}(v)=\{u\in V(H):uv\in E(H)\}$. Let $N_H[v]=\{v\}\cup N_{H}(v)$. The degree of the vertex $v$ is written as $d_{H}(v)$ or simply $d(v)$. $\Delta (H)$ is the maximum degree of $H$ and $\delta (H)$ is the minimum degree of $H$. Usually, a path of order $n$ is denoted by $P_n$, a cycle of order $n$ is denoted by $C_n$.  A star of order $n+1$ is denoted by $S_n$ ($n\geq 2$), and the vertex of degree larger than one is called the center vertex.
	The maximum number of edges in a matching of $H$ is called the matching number of $H$ and denoted  by $\alpha'(H)$. 
	For $U\subseteq V(H)$, let $H[U]$ be the subgraph of $H$ induced by $U$, $H-U$ be the graph obtained by deleting all the vertices in $U$ and their incident edges. 
	
	Given two graphs $G$ and $H$, the union of graphs $G$ and $H$ is the graph $G\cup H$ with vertex set $V(G)\cup V(H)$ and edge set $E(G)\cup E(H)$. The union of $k$ copies of $P_2$ is denoted by $kP_2$.
	The join of $G$ and $H$, denoted by $G\vee H$, is the graph obtained from $G\cup H$  by adding all edges between $V(G)$ and $V(H)$. The graph $K_p(i_1,i_2,\cdots,i_p)$ denotes the complete $p$-partite graph with parts of order $i_1,\,i_2,\,\cdots,\,i_p$. Denoted by $T_p(n)$, the $p$-partite Tur\'{a}n graph is the complete $p$-partite graph on $n$ vertices with the order of each partite set as equal as possible. 
	
		Given a family of graphs $\mathcal{L}$, a graph $H$ is $\mathcal{L}$-free if it does not contain any  graph $L\in \mathcal{L}$ as a subgraph. The Tur\'{a}n number, denoted by {\rm {ex}}$(n,\,\mathcal{L})$, is the maximum number of edges in a graph of order $n$ that is $\mathcal{L}$-free.
	The set of $\mathcal{L}$-free graphs of order $n$ with {\rm {ex}}$(n,\,\mathcal{L})$ edges is denoted by {\rm {EX}}$(n,\,\mathcal{L})$ and call a graph in {\rm {EX}}$(n,\,\mathcal{L})$ an extremal graph for $\mathcal{L}$.
	In 1966, Erd\H{o}s and Simonovits \cite{ref66} proved a classic theorem showing that the Tur\'{a}n number of a graph is closely related to the chromatic number. The chromatic number of $H$ is denoted by $\chi (H)$. 
	For a family of graphs $\mathcal{L}$, the subchromatic number of $\mathcal{L}$ is defined by $p(\mathcal{L})=$min$\{\chi (L):\,L\in \mathcal{L}\}-1$.

	\begin{theorem} [{Erd\H{o}s and Simonovits~\cite[]{ref66}}]\label{a}
		Given a family of graphs  $\mathcal{L}$, $p=p(\mathcal{L})$, if $p>0$, then 
		\begin{equation*}
			{\rm {ex}}(n,\,\mathcal{L})= \left(1-\frac{1}{p}\right)\binom{n}{2}+o(n^2).
		\end{equation*}
	\end{theorem}
Erd\H{o}s and Stone \cite{ref46} proved the following theorem, which shows that if the size of a graph satisfies some conditions, it  contains a Tur\'{a}n graph as a subgraph.
	\begin{theorem}[{Erd\H{o}s and Stone~\cite[]{ref46}}]\label{b}
	For all integers $p\geq 2$ and $N\geq 1$, and every $\epsilon >0$, there exists an integer $n_1$ such that every graph with $n>n_1$ vertices and at least $e(T_{p-1}(n))+\epsilon n^2$ edges contains $T_{p}(pN)$ as a subgraph.
	\end{theorem}

%	In 2009, based on Gallai's Lemma [2], Balachandran and Khare [3] gave a more 'structural' proof of this result. Hence they give a simple characterization of all the cases where the extremal graph is unique.

	 In 2003, Chen, Gould and Pfender \cite{ref03} determined the Tur\'{a}n numbers for $F_{k,r}$, a graph consists of $k$ complete graphs of order $r$ which intersect in exactly one common vertex. 
	In 2016, Hou, Qiu and  Liu \cite{ref16} determined the Tur\'{a}n numbers for intersecting odd cycles with the same length. Later, Hou, Qiu and Liu \cite{ref18} considered the Tur\'{a}n numbers for $H_{s,\,t}$, a graph consists of $s$ triangles and $t$ odd cycles with length at least 5 which intersect in exactly one common vertex. For an odd integer $t\geq 3$, the odd-ballooning of a graph $H$, denoted by $H(t)$, is the graph obtained from $H$ by replacing each edge of $H$ with an odd cycle of length at least $t$ where the new vertices of the odd cycles are all different.
	It is easy to see that $H_{s,\,t}$ can be seen as an odd-ballooning of a star $S_{s+t}$. In 2020, Zhu, Kang and Shan \cite{ref202} determined the Tur\'{a}n numbers for odd-ballooning of paths and cycles. Recently, Zhu and  Chen determined the Tur\'{a}n numbers for odd-ballooning of trees.
	In this paper, we determine the range of Tur\'{a}n numbers for odd-ballooning of general bipartite graphs when $t\geq 5$ by using progressive induction.

	A covering of a graph $H$ is a set of vertices which meets all edges of $H$. The minimum number of vertices in a covering of $H$ is denoted by $\beta (H)$. An independent covering of a bipartite graph $H$ is an independent set which meets all edges. The minimum number of vertices in an independent covering of a bipartite graph $H$ is denoted by $\gamma(H)$. 
	For any connected bipartite graph $H$, let $A$ and $B$ be its two color classes with $\lvert A\rvert \leq \lvert B\rvert $. Moreover, if $H$ is disconnected, we always partition $H$ into $A \cup B$ such that (1) $\lvert A\rvert $ is as small as possible; (2) min$\{d_H(x):x\in A\}$ is as small as possible subject to (1).
	In this paper, we study the Tur\'{a}n numbers for odd-ballooning of  bipartite graph $H$.
	\begin{lemma}[{Yuan~\cite[]{ref22}}]
		\label{i}
		Let $H$ be a bipartite graph, $V(H)=A\cup B$, then we have $\gamma(H)=\lvert A\rvert$ and each independent covering of $H$ contains either all the vertices of $A$ or all the vertices  of $B$.
	\end{lemma}

  Given a family of graphs $\mathcal{L}$, the following three parameters $q(\mathcal{L})$, $\mathcal{S}(\mathcal{L})$ and $\mathcal{B}(\mathcal{L})$ are proposed in \cite{ref22}. The \emph{independent covering number} $q(\mathcal{L})$ of $\mathcal{L}$ is defined as 
	\begin{equation*}
		q(\mathcal{L})={\rm min}\{\gamma (L):\;L\in \mathcal{L}\; {\rm is\; bipartite}\}.
	\end{equation*}
	The $independent \;covering\; family \; \mathcal{S}(\mathcal{L})$ of $\mathcal{L}$ is the family of independent coverings of bipartite graphs $L\in \mathcal{L}$ of order $q(\mathcal{L})$. The $subgraph \;covering\;family \;\mathcal{B}(\mathcal{L})$ of $\mathcal{L}\,$ is the set of subgraphs induced by a covering of $L\in \mathcal{L}$ with order less than $q(\mathcal{L})$. If $\beta (L)\geq q(\mathcal{L})$ for each $L\in \mathcal{L}$, then we set $\mathcal{B}(\mathcal{L})=\{K_{q(\mathcal{L})}\}$. 
	
	\begin{definition}[{Ni, Kang and Shan~\cite[]{ref201}}]\label{e}
		 Given a family of graphs $\mathcal{L}$, define $p=p(\mathcal{L})$. For any integer $p':\, 2\leq p'\leq p$, let $\mathcal{M}_{p'}(\mathcal{L})$ be the family of minimal graphs $M$ for which there exist an $L \in \mathcal{L}$ and a $t=t(L)$ such that there is a copy of $L$ in ${M^{'}} \vee {K_{p'-1}}(t,t, \cdots ,t)$ where ${M^{'}}=M \cup I_{t}$. We call this the p$'$-decomposition family of $\mathcal{L}$.
	\end{definition}
 For a bipartite graph $H$, we have $\chi(H(t))=3$. Therefore, in this paper, we mainly use 2-decomposition family of $H(t)$. 
 
 Given a graph $H$, by the definition of $\mathcal{M}_2(H(t))$, for any $M\in \mathcal{M}_2(H(t))$ there exist two independent sets $Y_1$, $Y_2$ such that there is a copy of $H(t)$ as a subgraph in $(M\cup Y_1)\vee Y_2$. Let $H_M$ be a copy of $H$ and $H_M$ satisfy that $H_M(t)\subseteq (M\cup Y_1)\vee Y_2$ is a copy of $H(t)$. $f$ is a bijection: $V(H_M)\to V(H)$ such that $uv\in E(H_M)$ if and only if $f(u)f(v)\in E(H)$. We may directly obtain the following lemma.
 
 \begin{proposition}\label{p}
 	Suppose $H$ is a  graph and $t\geq 3$ is an odd integer. For any $M\in \mathcal{M}_2(H(t))$, $M$ satisfies the following properties:
 	\begin{enumerate}[(i)]
 		\item $e(M)=e(H_M)=e(H)$;
 		\item each odd cycle of $H_M(t)$ expanded from an edge of $H_M$ contains exactly one edge in $M$;
 		\item $V(M)\subseteq V(H_{M}(t))$;
 		\item the vertex of $M$ which is in at least two odd cycles expanded from edges of $H_M$ is the vertex of $H_M$.
 	\end{enumerate}
 \end{proposition}

 In the following part of this paper, we always write  $\tilde{q}(H)=q(\mathcal{M}_2(H(t)))$, $\tilde{\mathcal{S}}(H)=\mathcal{S}(\mathcal{M}_2(H(t)))$, $\tilde{\mathcal{B}}(H)=\mathcal{B}(\mathcal{M}_2(H(t)))$, $k(H)=$min$\{d_{M}(x):x\in S, S\in \tilde{\mathcal{S}}(H)\}$, where $M\in \mathcal{M}_2(H(t))$ has the independent covering set  $S$. 
	
	Set $F(n,\,q)=I_{q-1}\vee T_{2}(n-q+1)$ and $f(n,\,q)=e(F(n,\,q))$. 
	For a family of graphs $\mathcal{L}$, denote by $\mathcal{F}(n,\,q,\,k,\,\mathcal{L})$, the set of graphs which are obtained by taking an $F(n,\,q)$, putting a copy of $K_{k,\,k}$ in one class of $T_2(n-q+1)$ and putting a copy of a member of {\rm {EX}}$(q-1,\,\mathcal{L})$ in $I_{q-1}$. 
	Our main results are as follows.
	\begin{theorem}\label{d}
		Let $H$ be a bipartite graph, $t\geq 5$ be an odd integer and  $n$ be a sufficiently large integer. Then
		\begin{equation*}
			f(n,\,\tilde{q}(H))+ {\rm {ex}}(\tilde{q}(H)-1,\,\tilde{\mathcal{B}}(H))\leq {\rm {ex}}(n,\,H(t)) \leq f(n,\,\tilde{q}(H))+{\rm {ex}}(\tilde{q}(H)-1,\,\tilde{\mathcal{B}}(H))+(k(H)-1)^2.
		\end{equation*}
		Moreover, if  ${\rm {ex}}(n,\,H(t)) = f(n,\,\tilde{q}(H))+{\rm {ex}}(\tilde{q}(H)-1,\,\tilde{\mathcal{B}}(H))+(k(H)-1)^2$ holds, then the graphs in $\mathcal{F}(n,\,\tilde{q}(H),\,k(H)-1,\,\tilde{\mathcal{B}}(H))$ are the only extremal graphs for $H(t)$.
	\end{theorem}

	\section{Characterizations of $\mathcal{M}_2(H(t))$ and $\tilde{\mathcal{B}}(H)$} 
	Given a graph $H$, the {\emph{vertex division}} on some non-pendent vertex $v$ of $H$  is defined as follows:  $v$ is replaced by an independent set $\{v',\,v_1,\,v_2,\,\cdots,\,v_m\}$ ($1\leq m\leq d_H(v)-1$) in which  $v_i$ ($1\leq i\leq m$) is adjacent to exactly one distinct vertex in $N_H(v)$ and $v'$ is adjacent to the remaining neighbors of $v$ in $H$. In particular, if $m=d_{H}(v)-1$, it is called vertex split in \cite{ref13,ref201}. Denote by $\mathcal{D}(H)$, the family of graphs which can be obtained by applying vertex division on some vertex set $U\subseteq V(H)$. An isolated edge is an edge whose endpoints has degree 1.
	Lemma \ref{f} shows that the 2-decomposition family of odd-ballooning for any graph $H$ is actually the family of graphs obtained from dividing some vertices of $H$.
	%If $\chi(H[U])\leq i $, we can denote the family of graphs by $\mathcal{D}_i(H)$. If $\chi(H[U])=1 $, we can denote the family of graphs by $\mathcal{D}^*(H)$. It is easy to see that, when $i\geq 2$, $\mathcal{D}^*(H)\subseteq \mathcal{D}_i(H)\subseteq \mathcal{D}(H)$.

	\begin{lemma}\label{f}
		Let $H$ be any graph and $t\geq 5$ be an odd integer, then $\mathcal{M}_2(H(t))\subseteq \mathcal{D}(H)$ holds.
	\end{lemma}
	\begin{proof}
		To prove Lemma \ref{f}, we may show that any graph in $\mathcal{M}_2(H(t))$ can be obtained by using vertex division on some vertices of $V(H)$.
		For any graph $M\in \mathcal{M}_2(H(t))$,
		by Proposition \ref{p} (i), we have $e(M)=e(H)$.
		$H_M$ is a copy of $H$ and there exist two independent sets $Y_1$ and $Y_2$ such that $H_M(t)\subseteq (M\cup Y_1)\vee Y_2$ holds.
	 Furthermore, $d_M(v)\geq 1$ holds for any vertex $v\in V(M)$. If there is an isolated vertex $v$ in $M$, then we may add a  vertex  $v'$ in $Y_1$ to replace $v$, and we have $H_M(t)\subseteq ((V(M)\backslash\{v\})\cup (Y_1\cup \{v'\}))\vee Y_2$ which contradicts the minimality of $M$.
		
		For any vetrex $v$ of $M$, first, we suppose $v\in V(M)\cap V(H_M)$.
	If $d_M(v)> d_{H_M}(v)$, then $d_M(v)-d_{H_M}(v)$ edges are not  in $H_M(t)$ and it  contradicts  the minimality of $M$. Hence we have $d_M(v)\leq d_{H_M}(v)$. 
	When  $d_M(v)<d_{H_M}(v)$, then $Y_2$ contains $x$ neighbors ($d_{H_M}(v)-d_M(v)\leq x\leq d_{H_M}(v)$) of $v$ in $H_M$. 
		Each edge between $Y_2$ and $v$ can be expanded into an odd cycle by using one edge in $M$. By the minimality of $M$, in $M$ there is a star $S_{d_M(v)}$ and $d_{H_M}(v)-d_M(v)$ distinct edges can be used to obtain an $S_{d_{H_M}(v)}(t)$. Since the new vertices of $S_{d_{H_M}(v)}(t)- V(S_{d_{H_M}(v)})$ are all different, these $d_{H_M}(v)-d_M(v)$ edges are independent. Therefore,
		to obtain $M$, we may divide the vertex $f(v)$ of $H$ into a vertex with degree $d_M(v)$ and an independent set of order $d_{H}(v)-d_M(v)$.
		When  $d_M(v)=d_{H_M}(v)$, the adjacency relation of $v$ in $M$ is the same as $f(v)$ in $H$.
		
		Now we suppose $v\in V(M)\cap (V(H_M(t))\backslash V(H_M))$. The fact that the new vertices of odd cycles are all different implies the edges incident to $v$ are in the same odd cycle.
		By Proposition \ref{p} we have $d_M(v)=1$. Suppose $v$ is in an odd cycle expanded from the edge $uw$ of $H_M$. Then the edge $uw$ is not in $M$, otherwise this odd cycle has two edges in $M$, a contradiction. 
		Suppose $uw$ is between $M$ and $Y_2$, $u\in V(M)$, $w\in Y_2$. Since $N_{H_M}(w)\subseteq V(M)\cup Y_1$, each odd cycle of $S_{d_{H_M}(w)}(t)$ contains exactly one  edge in $M$, there are $d_{H_M}(w)$ independent edges in $M$. Thus, to obtain $M$, we may use vertex split on $f(w)$ of $H$. If $uw$ is between $Y_1$ and $Y_2$, $u\in Y_1$, $w\in Y_2$, then $N_{H_M}(u)\subseteq Y_2$. Noting that $S_{d_{H_M}(u)}(t)$ spanned by $N_{H_M}[u]$ in $H_M(t)$ contains exactly $d_{H_M}(u)$ independent edges in $M$.
		We may deduce that the edge which contains vertex $v$ is an isolated edge in $M$. Hence  to obtain $M$ we may use vertex split on both $f(u)$ and $f(w)$ of $H$. 

	Therefore, we have $M\in \mathcal{D}(H)$. As $M$ is arbitrary, we have 
		$\mathcal{M}_2(H(t))\subseteq \mathcal{D}(H)$. 
	\end{proof}

\begin{lemma}\label{2.3}
	Let $H$ be a bipartite graph, $t\geq 5$ be an odd integer. Then $e(H)P_2\in \mathcal{M}_2(H(t))$ holds.
\end{lemma}
\begin{proof}
	Suppose $Y_1$ and $Y_2$ are two independent sets and large enough.  Let $H'$ be a copy of $H$, $V(H')=A'\cup B'$. $A'$ corresponds to $A$, $B'$ corresponds to $B$. Let $A'\subseteq V(e(H)P_2)$ and they are independent in $e(H)P_2$, $B'\subseteq Y_2$, then we have $H'\subseteq e(H)P_2\vee Y_2$.
	In the graph $(e(H)P_2\cup Y_1)\vee Y_2$,
	 the edge of $H'$  can be expanded into an odd cycle by using an edge in $e(H)P_2$ and some vertices of $Y_1$ and $Y_2$, then we have $H'(t)\subseteq (e(H)P_2\cup Y_1)\vee Y_2$. Therefore, $e(H)P_2$ contains a subgraph in $\mathcal{M}_2(H(t))$. Moreover, Proposition \ref{p}  (i) implies that $e(H)P_2\in \mathcal{M}_2(H(t))$ holds.
\end{proof}
%\begin{lemma}\label{g}
	%\begin{enumerate}[(i)]
	%	\item  For any bipartite graph $H$, if there is an independent covering $S\in \mathcal{S}_H$ such that at least one vertex of $S$ is obtained by dividing vertices of $H$, then $k(H)=1$.
	%	\item If $H$ is bipartite with $\tilde{q}(H)<\lvert A\rvert $, then $k(H)=1$.
%	\end{enumerate}
%\end{lemma}
%\begin{proof}
%	Denote by  $\mathcal{M}=\mathcal{M}_2(H(t))$, $H_S\in \mathcal{M}$ be a graph with $q(H_S)=\tilde{q}(H)$ and $S$ be an independent covering of $H_S$ of order $\tilde{q}(H)$. 
%	\begin{enumerate}[(i)]
	%	\item If there is a vertex $x\in S$ obtained by applying vertex split, it is easy to see that $k(H)=1$.
	%	If there is a vertex $x\in S$ obtained by applying vertex division not split on the vertex of $H$, we may suppose $x\in V(H)$ is divided into $x,\,x_1,\,x_2,\cdots, x_m$ ($1\leq m< d_H(x)-1$) where $d_{H_S}(x)\geq 2$, $d_{H_S}(x_i)=1$ ($1\leq i\leq m$).
	%	From the proof of Lemma \ref{f}, the neighbors of $x_1,\,x_2,\cdots, x_m$ are split or with degree 1 in $H$, then $H_S$ contains $m$ independent edges. Since $k(H)$ is the minimum degree of $S$, we have $k(H)=1$. 
	%	\item If $H$ is a bipartite graph with $\tilde{q}(H)<\lvert A\rvert$, then by Lemma \ref{i} we know that there is an $x\in S$ which is obtained by dividing a vertex in $H$. Thus by (i) we have $k(H)=1$.
	%\end{enumerate}
%\end{proof}

%	For some graph $H$, by the proof of Lemma \ref{f}, we may determine the $\tilde{\mathcal{B}}(H)$.
	Let $H$ be a bipartite graph, $t\geq 5$ be an odd integer. Denoted by $\mathcal{N}(H(t))\subseteq \mathcal{M}_2(H(t))$ is the family of graphs $M$ with $\gamma(M)=\tilde{q}(H)$.
	\begin{lemma}\label{3}
		Let $H$ be a bipartite graph,  $M\in \mathcal{M}_2(H(t))$ be a graph with an independent covering $S\in \tilde{\mathcal{S}}(H)$.
		If  $S$  contains a vertex with degree one in $M$, then $\mathcal{N}(H(t))$ contains a graph with an isolated edge.
	\end{lemma}
	
	\begin{proof}
		Let  $H_M$ be a copy of $H$ and there exist two independent sets $Y_1$ and $Y_2$ such that $H_M(t)\subseteq (M\cup Y_1)\vee Y_2$ holds. If there is an isolated edge in $M$, the conclusion holds. Now suppose  there is no isolated edge in $M$.
	
	Let $v'\in S$ and $N_M(v')=\{u\}$ hold. If $d_M(u)=1$, then the edge $uv'$ is an isolated edge in $M$, a contradiction. Thus we have $d_M(u)\geq 2$,  then by Proposition \ref{p} (iv), $u$ is a vertex of $H_M$. Let $M'$ be the graph obtained from dividing vertex $u$ into  $\{u',\,u''\}$, $u'$ is adjacent to $v'$ with degree one, $u''$ is adjacent to the remaining neighbors of the original vertex $u$ in $M$. In the graph $(M'\cup Y_1)\vee Y_2$, there is a vertex $w\in Y_2$ adjacent to $u''$, the edge $u''w$ can be expanded into an odd cycle by using the edge $v'u'$ in $M'$. As there is an $H_M(t)$ in $(M\cup Y_1)\vee Y_2$, and the vertices in $M'$ except $v'$, $u'$, $u''$ have the same neighbors of the vertices in $M$, $(M'\cup Y_1)\vee Y_2$ contains a copy of $H(t)$ as a subgraph. It is easy to see that $e(M)=e(M')$ and there is an isolated edge in $M'$, thus we have $M'\in \mathcal{M}_2(H(t))$. 	As $S$ is the independent covering of $M$, $v'$ is a vertex in $S$, we have $u\notin S$. When we divide the vertex $u$ of $M$,  $S$ is also an independent covering of $M'$. From the definition of $\tilde{\mathcal{S}}(H)$, we have $\lvert S\rvert =\tilde{q}(H)$ and then $\gamma (M')\leq \tilde{q}(H)$. Because $\tilde{q}(H)$ is the minimum size of the independent covering of graphs in $\mathcal{M}_2(H(t))$ and $M'\in \mathcal{M}_2(H(t))$, we may deduce $\gamma (M')= \tilde{q}(H)$ and then $M'\in \mathcal{N}(H(t))$. 
	\end{proof}
	
	\begin{lemma}\label{2.4}
		Let $H$ be a bipartite graph with $V(H)=A\cup B$, $M\in \mathcal{M}_2(H(t))$ be a graph with an independent covering $S\in \tilde{\mathcal{S}}(H)$ and min $\{d_M(x):x\in S\}=k(H)$.
		 If each vertex of $S$ has degree at least 2 in $M$,  then $\lvert A\rvert =\tilde{q}(H)$ and min $\{d_H(x):x\in A\}=k(H)$.
	\end{lemma}
\begin{proof}
		Let  $H_M$ be a copy of $H$ and there exist two independent sets $Y_1$ and $Y_2$ such that $H_M(t)\subseteq (M\cup Y_1)\vee Y_2$ holds. 	
	By Proposition \ref{p} (iv), we have $S\subseteq V(H_M)$. Suppose $e$ is an edge of $H_M$. If $e\in E(M)$, then $e$ is covered by $S$. If $e$ is between $Y_1$ and $Y_2$, there is an isolated edge in $M$, a contradiction.
	If $e$ is between $M$ and $Y_2$, then let $xy=e$, $x\in V(M)$, $y\in Y_2$. Since there is no isolated edge in $M$, there is an edge $xy''\in E(M)$ to be used to expand $e$ into an odd cycle, where $N_M(y'')=\{x\}$, then we have $x\in S$. Hence the edge $e$ is covered by $S$. Therefore, $S$ is an independent covering of $H_M$ and then $\lvert A\rvert =\tilde{q}(H)$ holds. 
	
	Suppose $w$ is a vertex in  $S$. If $d_M(w)<d_{H_M}(w)$, $N_{H_M}(w)$ contains a vertex $w'$ in $Y_2$. The edge $ww'$ can be expanded into an odd cycle by using an edge $e$ in $M$ and $e$ is in $H_{M}(t)-V(H_M)$. The fact that the new vertices of the odd cycles are all distinct implies $e$ is an isolated edge in $M$, a contradiction. If  $d_M(w)>d_{H_M}(w)$, then $M$ contains $d_M(w)-d_{H_M}(w)$ edges not in $H_M(t)$, a contradiction to the minimality of $M$. So $d_M(w)=d_{H_M}(w)$ holds. As $w$ is arbitrary, we have min $\{d_H(x):x\in A\}=$ min $\{d_{H_M}(x):x\in f^{-1}(A)\}=$ min $\{d_M(x):x\in S\}=k(H)$.
			\end{proof}

	\textbf{Example 1} 
	For the star $S_a$, by Lemma \ref{f}, we have $\mathcal{M}_2(S_a(t))\subseteq \mathcal{D}(S_a)$. On the other hand, for any $M\in \mathcal{D}(S_a)$, we have  $M=S_x\cup (a-x)P_2$ ($1\leq x\leq a$). $Y_1$ and $Y_2$ are two independent sets and large enough. In the graph $(M\cup Y_1) \vee Y_2$,
	by using the vertices of $Y_1$ and $Y_2$, different edges of $M$ can be used to expand  different odd cycles of $S_a(t)$. Since $e(M)=e(S_a)$, we have $S_a(t)\subseteq (M\cup Y_1) \vee Y_2$ and  $M\in \mathcal{M}_2(S_a(t))$. As $M$ is arbitrary, we  have $ \mathcal{D}(S_a) \subseteq\mathcal{M}_2(S_a(t))$. Hence we have $\mathcal{M}_2(S_a(t))= \mathcal{D}(S_a)$.
	Then we may imply  $\tilde{q}(S_a)=1$, $\tilde{\mathcal{S}}(S_a)$ is the center vertex of the star and $k(S_a)=a$. For any graph $M\in \mathcal{M}_2(S_a(t))$, $\beta(M)\geq \tilde{q}(S_a)$, hence we have $\tilde{\mathcal{B}}(S_a)=\{K_1\}$.

	\textbf{Example 2}
	For the path $P_{m+1}$, by Lemma \ref{f}, $\mathcal{M}_2(P_{m+1}(t))\subseteq \mathcal{D}(P_{m+1})$. On the other hand, for any $M\in \mathcal{D}(P_{m+1})$, $M$ is a union of some paths and $e(M)=m$. Set $Y_1$ and $Y_2$ be two independent sets and large enough. In the graph $(M\cup Y_1) \vee Y_2$, by using the vertices of $Y_1$ and $Y_2$, different edges of $M$ can be used to expanded different odd cycles of $P_{m+1}(t)$. Since $e(M)=m$, we have $P_{m+1}(t)\subseteq (M\cup Y_1) \vee Y_2$ and $M\in \mathcal{M}_2(P_{m+1}(t))$. Hence we have 
	$\mathcal{D}(P_{m+1})\subseteq \mathcal{M}_2(P_{m+1}(t))$. Therefore $\mathcal{D}(P_{m+1})= \mathcal{M}_2(P_{m+1}(t))$ holds.
	 When $m$ is even, $\tilde{q}(P_{m+1})=\frac{m}{2}$, $\tilde{\mathcal{S}}(P_{m+1})$ consists of the independent coverings of the graphs in $\mathcal{M}_2(P_{m+1}(t))$ such that each component of them is a path with even edges, and $k(P_{m+1})=2$. For any graph $M\in \mathcal{M}_2(P_{m+1}(t))$, $\beta(M)\geq \tilde{q}(P_{m+1})$, hence we have $\tilde{\mathcal{B}}(P_{m+1})=\{K_{\frac{m}{2}}\}$.
	When $m$ is odd,  we have $\tilde{q}(P_{m+1})=\frac{m+1}{2}$. Then $\tilde{\mathcal{S}}(P_{m+1})$ consists of the independent coverings of the graphs in $\mathcal{M}_2(P_{m+1}(t))$ such that each component of them is a path with even edges except one component is a path with odd edges and  $k(P_{m+1})=1$. For  any graph $M\in \mathcal{M}_2(P_{m+1}(t))$, $\beta(M)\geq \tilde{q}(P_{m+1})$, hence we have $\tilde{\mathcal{B}}(P_{m+1})=\{K_{\frac{m+1}{2}}\}$.

%	\textbf{Example 3} 
%	For the even cycle $C_{m}$, by Lemma \ref{f}, $\mathcal{M}_2(C_{m}(t))\subseteq \mathcal{D}(C_{m})$. On the other hand, for any $M\in \mathcal{D}(C_{m})$, $M$ is a union of some paths and $e(M)=m$ or $M=C_m$ holds. 
%	Set $Y_1$ and $Y_2$ be two independent sets and large enough. In the graph $(M\cup Y_1) \vee Y_2$, by using the vertices of $Y_1$ and $Y_2$, different edges of $M$ can be used to expanded into different odd cycles of $C_{m}(t)$. Since $e(M)=m$, we have $C_{m}(t)\subseteq (M\cup Y_1) \vee Y_2$ and $M\in \mathcal{M}_2(C_{m}(t))$. Hence we have 
%	$\mathcal{D}(C_{m})\subseteq \mathcal{M}_2(C_{m}(t))$. Therefore $\mathcal{D}(C_{m})= \mathcal{M}_2(C_{m}(t))$ holds.
%	Then we have $\tilde{q}(C_{m})=\frac{m}{2}$, $\tilde{\mathcal{S}}(C_{m})$ consists of the independent coverings of the graphs in $\mathcal{M}_2(C_{m}(t))$ such that each component of them is a path with even edges or an even cycle, and $k(C_{m})=2$. For any graph $M\in \mathcal{M}_2(C_{m}(t))$, $\beta(M)\geq \tilde{q}(C_{m})$, hence we have $\tilde{\mathcal{B}}(C_{m})=\{K_{\frac{m}{2}}\}$.

	As stars, paths and even cycles satisfy the conditions of Theorem \ref{d}, Theorem \ref{d} implies the Tur\'{a}n numbers for odd-ballooning of stars, paths and even cycles. 
\begin{corollary}[{Hou, Qiu and Liu~\cite[]{ref18}}]\label{n}
	If $n$ is sufficiently large and $t\geq 5$ is an odd integer, then  
	\begin{equation*}
		{\rm{ex}}(n,\,S_{a}(t))=e(T_2(n))+(a-1)^2
	\end{equation*}
	holds and the only extremal graph for $S_{a}(t)$ is the graph obtained from $T_2(n)$ by putting a copy of $K_{a-1,\,a-1}$ in one class of $T_2(n)$.
\end{corollary}
\begin{proof}
	Let $F$ be the graph obtained by putting a copy of $K_{a-1,\,a-1}$ in one class of $T_2(n)$. From Example 1, we have  $\mathcal{M}_2(S_a(t))=\{S_x\cup (a-x)P_2\;|\, 1\leq x\leq a\}$. If $S_{a}(t)\subseteq F$, then we have $S_{a}(t)\subseteq \left(K_{a-1,\,a-1}\cup I_m\right)\vee I_m$ where $m=\lceil\frac{n}{2} \rceil -2a$. So  $K_{a-1,\,a-1}$ contains a subgraph as a copy of a member of $\mathcal{M}_2(S_a(t))$. For any graph $M\in\mathcal{M}_2(S_a(t))$, $M=S_x\cup (a-x)P_2$ ($1\leq x\leq a$), we have $x+\alpha'(M-V(S_x))=a$.
	However,
	$x+\alpha'(K_{a-1,\,a-1}-V(S_x))=a-1<a$. So $F$ is $S_{a}(t)$-free and 
	\begin{equation*}
		{\rm {ex}}(n,\,S_{a}(t))\geq e(T_2(n))+(a-1)^2.
	\end{equation*}
	
	On the other hand, from Example 1, we have $\tilde{q}(S_a)=1$, $k(S_a)=a$ and $\tilde{\mathcal{B}}(S_a)=\{K_1\}$. By applying Theorem \ref{d}, we have
	\begin{equation*}
		{\rm {ex}}(n,\,S_{a}(t))\leq f(n,\,1)+(a-1)^2= e(T_2(n))+(a-1)^2.
	\end{equation*}
	Therefore, ${\rm {ex}}(n,\,S_{a}(t))=e(T_2(n))+(a-1)^2$ holds. Noting that $\mathcal{F}(n,1,a-1,K_1)=\{F\}$ holds, hence $F$ is the only extremal graph for $S_{a}(t)$.
\end{proof}
	
	\begin{corollary}[{Zhu, Kang and Shan~\cite[]{ref202}}]\label{c1}
		Let $n$ be a sufficiently large integer, $t$ be an odd integer at least $5$. We have the following:
		\begin{enumerate}[(i)]
			\item if $m$ is even, let $d=\frac{m}{2}$, then
			\begin{equation*}
				{\rm {ex}}(n,\,P_{m +1}(t))=e(T_2(n-d+1)\vee K_{d-1})+1
			\end{equation*}
			holds and the only extremal graph for $P_{m +1}(t)$ is the graph obtained from $T_2(n-d+1)\vee K_{d-1}$ by putting an edge in one class of $T_2(n-d+1)$.
			\item if $m$ is odd, let $d=\frac{m+1}{2}$, then
			\begin{equation*}
				{\rm {ex}}(n,\,P_{m +1}(t))=e(T_2(n-d+1)\vee K_{d-1})
			\end{equation*}
			holds and $T_2(n-d+1)\vee K_{d-1}$ is the unique extremal graph for $P_{m +1}(t)$.
		\end{enumerate}
	\end{corollary}
	\begin{proof}
	(i) When $m$ is even, let $F$ be obtained from $T_2(n-d+1)\vee K_{d-1}$ by putting an edge in one class of $T_2(n-d+1)$. In $F$, each vertex of $K_{d-1}$ is contained in at most two odd cycles of $P_{m+1}(t)$ and the edge in one class of $T_2(n-d+1)$ can be contained in only one odd cycle of $P_{m+1}(t)$. Therefore, in $F$, 
			the number of odd cycles in the odd-ballooning of  $P_{m+1}$ is at most $2(d-1)+1=m-1<m$. So $F$ is $P_{m+1}(t)$-free and 
			\begin{equation*}
				{\rm {ex}}(n,\,P_{m +1}(t))\geq e(T_2(n-d+1)\vee K_{d-1})+1=e(F).
			\end{equation*}
			On the other hand, from Example 2, we have $\tilde{q}(P_{m+1})=\frac{m}{2}$, $k(P_{m+1})=2$, $\tilde{\mathcal{B}}(P_{m+1})=\{K_{\frac{m}{2}}\}$. By applying Theorem \ref{d}, we have 
			\begin{equation*}
				{\rm {ex}}(n,\,P_{m+1}(t))\leq f'(n,\,d)+{\rm {ex}}(d-1,\,K_d)+1=e(T_2(n-d+1)\vee K_{d-1})+1.
			\end{equation*}
			Therefore ${\rm {ex}}(n,\,P_{m+1}(t))=e(T_2(n-d+1)\vee K_{d-1})+1$ holds. Noting that $\mathcal{F}(n,d,1,K_{d-1})=\{F\}$, hence  
			 $F$ is  the only extremal graph for $P_{m +1}(t)$.
			 
		(ii) When $m$ is odd, from Example 2, we have $\tilde{q}(P_{m+1})=\frac{m+1}{2}$, $k(P_{m+1})=1$, $\tilde{\mathcal{B}}(P_{m+1})=\{K_{\frac{m+1}{2}}\}$. By applying Theorem \ref{d}, the lower bound and the upper bound of the inequality in Theorem \ref{d} are the same, we have  
			$${\rm {ex}}(n,\,P_{m+1}(t))=f(n,\,d)+{\rm {ex}}(d-1,\,K_d)=e(T_2(n-d+1)\vee K_{d-1})$$ and  $T_2(n-d+1)\vee K_{d-1}$ is the unique extremal graph for $P_{m+1}(t)$.
	\end{proof}
	
	Using the similar arguments as the proof of Corollary \ref{c1} (i), we have the following corollary.
	\begin{corollary}[{Zhu, Kang and Shan~\cite[]{ref202}}]
		Given an even integer $m\geq4$, an odd integer $t\geq 5$ and  a sufficiently large integer $n$,  we have 
		${\rm {ex}}(n,\,C_{m}(t))=e(T_{2}(n-d+1)\vee K_{d-1})+1$ where $d=\frac{m}{2}$. The only extremal graph for $C_{m}(t)$ is the graph obtained from $T_{2}(n-d+1)\vee K_{d-1}$ by putting an edge in one class of $T_{2}(n-d+1)$.
	\end{corollary}
	
	Let $T$ be a tree and $V(T)=A\cup B$. An odd-ballooning $T(t)(t\geq 3)$ of $T$ is good if all edges which are expanded into triangles are the edges who have one endpoint with degree one and the non-leaf vertices are in $A$.
	Recently, Zhu \cite{ref221} gave the exact value of ex$(n,T(t))(t\geq 3)$ when $T(t)$ is a good odd-ballooning of $T$. As $\mathcal{M}_2(T(t))\subseteq \mathcal{D}(T)$, when $T(t)$ is a good odd-ballooning, we have $\tilde{q}(T)=\lvert A\rvert$. Hence if $A$ has a vertex with degree one in $T$, we have the following corollary.
	
	\begin{corollary}[{Zhu and Chen~\cite[]{ref221}}]
		Let $T$ be a tree with $V(T)=A\cup B$ and $a=\lvert A\rvert$. Suppose $T(t)$ is a good odd-ballooning of  $T$ where $t\geq 5$ is an odd cycle. If $A$ contains a vertex $u$ with $d_T(u)=1$, then
		\begin{equation*}
			{\rm{ex}}(n,T(t))=f(n,a)+{\rm{ex}}(a-1,\tilde{\mathcal{B}}(T)).
		\end{equation*}
	Moreover the extremal graphs for $T(t)$ are in $\mathcal{F}(n,a,0,\tilde{\mathcal{B}}(T))$.
	\end{corollary}

	\section{Proof of Theorem \ref{d}}
	This section is devoted to the proof of Theorem \ref{d}. 
    First, we prove the lower bound of Theorem \ref{d}.
	
	\begin{lemma}\label{l}
		Let  $H$ be a bipartite graph, $n$ be a sufficiently large integer, $t\geq 5$ be an odd integer, then we have
		\begin{equation*}
			{\rm {ex}}(n,\,H(t))\geq f(n,\,\tilde{q}(H)) +{\rm {ex}}(\tilde{q}(H)-1,\,\tilde{\mathcal{B}}(H)).
		\end{equation*}
	\end{lemma}
	
	\begin{proof}
		If there is an $H(t)\subseteq F \in  \mathcal{F}(n,\,\tilde{q}(H),\,0,\,\tilde{\mathcal{B}}(H))$ i.e.,
		\begin{equation*}
			\begin{split}
				H(t)\subseteq (Q\vee I_m)\vee I_m\subseteq ((Q\vee I_m)\cup I_m)\vee I_m,
			\end{split}
		\end{equation*}
		where $Q\in {\rm {EX}}(\tilde{q}(H)-1,\,\tilde{\mathcal{B}}(H))$ and $m=\lceil \frac{n-\tilde{q}(H)+1}{2}\rceil $.
		By the definition of $\mathcal{M}_2(H(t))$,
		$(G\cup I_m)\vee I_m$ contains a copy of $H(t)$, then $G$ contains a subgraph as a member of  $\mathcal{M}_2(H(t))$.
		Thus $Q\vee I_m$ contains a subgraph as a member, say $M$ of $\mathcal{M}_2(H(t))$. Since $I_m$ is an independent set, $Q$ contains a subgraph induced by a covering of $M$. However, when $\beta(M)<\tilde{q}(H)$ holds,  since the graphs in ${\rm {EX}}(\tilde{q}(H)-1,\,\tilde{\mathcal{B}}(H))$ are $\tilde{\mathcal{B}}(H)$-free, $Q$ contains no subgraph induced by a covering of $M$; when  $\beta(M)\geq \tilde{q}(H)$, since the order of $Q\in {\rm {EX}}(\tilde{q}(H)-1,\,\tilde{\mathcal{B}}(H))$ is $\tilde{q}(H)-1$, $Q$ contains no subgraph induced by a covering of $M$.
		This contradiction shows any graph in  $ \mathcal{F}(n,\,\tilde{q}(H),\,0,\,\tilde{\mathcal{B}}(H))$ is $H(t)$-free.
	
	Note that for any $F$ in $ \mathcal{F}(n,\,\tilde{q}(H),\,0,\,\tilde{\mathcal{B}}(H))$, we have $e(F)=f(n,\tilde{q}(H))+{\rm {ex}}(\tilde{q}(H)-1,\,\tilde{\mathcal{B}}(H))$. 
	 Therefore, ${\rm {ex}}(n,\,H(t))\geq f(n,\,\tilde{q}(H)) +{\rm {ex}}(\tilde{q}(H)-1,\,\tilde{\mathcal{B}}(H))$ holds.
\end{proof}
	To prove Theorem \ref{m}, we need the following lemmas. $c(H)$ is the number of components of $H$.
	\begin{lemma}[{Hou, Qiu and Liu~\cite[]{ref16}}]\label{j}
		Let $H$ be a graph with no isolated vertex. If $\Delta (H)\leq 2$, then 
		\begin{equation*}
			\alpha' (H)\geq \frac{\nu (H)-c(H)}{2}.
		\end{equation*}
	\end{lemma}
	
	\begin{lemma}[{Hou, Qiu and Liu~\cite[]{ref16}}]\label{k}
		Let $H$ be a graph with no isolated vertex. If for all $x\in V(H)$, $d(x)+	\alpha' (H-N[x])\leq k$, then $e(H)\leq k^2$. Moreover, the equality holds if and only if $H=K_{k,\,k}$.
	\end{lemma}

Define $\varphi(\alpha',\,\Delta)$=max$\{e(H):\alpha' (H)\leq \alpha',\,\Delta(H)\leq \Delta\}$. Chv\'{a}tal and Hanson \cite{ref76} proved the following theorem which is useful to estimate the number of edges of a graph with restricted degrees and matching number.
\begin{lemma}[{Chv\'atal and Hanson~\cite[]{ref76}}]\label{c}
	 For any graph $H$ with maximum degree $\Delta \geq 1$ and matching number $\alpha'\geq 1$, we have 
	\begin{equation*}
		e(H)\leq \varphi(\alpha' ,\,\Delta)=\alpha' \Delta +\Big {\lfloor} \frac{\Delta}{2}\Big {\rfloor} \Big{\lfloor} \frac{\alpha'}{\lceil \Delta/2 \rceil} \Big{\rfloor} \leq \alpha'\Delta+\alpha'.
	\end{equation*}
\end{lemma}

%By the formula of $\varphi(\alpha',\,\Delta)$ we may imply that 
%\begin{equation*}
%	\varphi (\alpha'_1,\,\Delta)+\varphi (\alpha' _2,\,\Delta)\leq \varphi(\alpha'_1+\alpha'_2,\,\Delta).
%\end{equation*}

	\begin{theorem}\label{m}
		Let $n,\;b$ be sufficiently large positive integers, $H$ be a bipartite graph.
		Let $G$ be a  graph of order $n$ with a partition of vertices into three parts $V(G)=V_0\cup V_1\cup V_2$ satisfying the following conditions:
		\begin{enumerate}[(i)]
			\item there exist $V'_{1}\subseteq V_1$, $V'_{2}\subseteq V_2$ such that $G[V'_{1}\cup V'_{2}]=T_2(2b)$;
			\item $\lvert V_0\rvert =\tilde{q}(H)-1$ and each vertex of  $V_0$ is adjacent to each vertex of $T_2(2b)$;
			\item each vertex of $V''_i=V_i\backslash V'_i$  is adjacent to at least $c_0 \lvert V'_{3-i} \rvert $ ($\frac{1}{2}<c_0\leq 1$) vertices of $V'_{3-i}$, and is not adjacent to any vertex of $V'_{i}$ ($i=1,\,2$).
		\end{enumerate}
If $G$ is  $H(t)$-free, then $$e(G)\leq  f(n,\,\tilde{q}(H))+{\rm {ex}}(\tilde{q}(H)-1,\,\tilde{\mathcal{B}}(H))+(k(H)-1)^2$$ and the equality holds if and only if $G\in \mathcal{F}(n,\,\tilde{q}(H),\,k(H)-1,\,\tilde{\mathcal{B}}(H))$.
	\end{theorem}

\begin{proof}
	Let $M$ be a graph in $\mathcal{M}_2(H(t))$ with an  independent set $S\in \tilde{\mathcal{S}}(H)$ and min$\{d_M(x):x\in S\}=k(H)$. $\tilde{\mathcal{B}}(H)$ is the subgraph covering family of $ \mathcal{M}_2(H(t))$.
	
	If $G[V_0]$ contains a graph in $\tilde{\mathcal{B}}(H)$,  there exists an $M\subseteq G[V_0\cup V'_1]$  such that $M\in \mathcal{M}_2(H(t))$ holds. Then we have  $H(t)\subseteq G[V_0\cup V'_1\cup V'_2 ]\subseteq G$, which is a contradiction. Thus, $e(G[V_0])\leq {\rm {ex}}(\tilde{q}(H)-1,\,\tilde{\mathcal{B}}(H))$.

	If $S$ contains a vertex with degree one in $M$,  by Lemma \ref{3}, we know that  $\mathcal{M}_2(H(t))$ contains a graph $M'$ with an isolated edge, say $uv$, and $\gamma(M')=\tilde{q}(H)$, $\gamma(M'-uv)=\tilde{q}(H)-1$. And in this case, we have $k(H)=1$. If there is an edge in $G[V''_1]$, noting that  $\lvert V_0\rvert =\tilde{q}(H)-1$ and each vertex of $V_0$ is adjacent to each vertex of $V'_1$, then $G[V_0\cup V_1]$ contains a copy of $M'$. The number of vertices in $V'_2$ adjacent to  $V''_1$ is sufficiently large.
	Then we have a copy of $H(t)\subseteq G[V_0\cup V_1\cup V'_2]\subseteq G$ which is a contradiction. Therefore there is no edge in $G[V_1]$. Similarly, there is no edge in $G[V_2]$. Then $e(G)\leq f(n,\,\tilde{q}(H))+{\rm {ex}}(\tilde{q}(H)-1,\,\tilde{\mathcal{B}}(H))$ holds. Together with Lemma \ref{l}, $e(G)= f(n,\,\tilde{q}(H))+{\rm {ex}}(\tilde{q}(H)-1,\,\tilde{\mathcal{B}}(H))$ holds and $G$ is in  $\mathcal{F}(n,\,\tilde{q}(H),\,0,\,\tilde{\mathcal{B}}(H))$.
	
	Now suppose that each vertex of $S$ has degree at least 2 in $M$, then by Lemma \ref{2.4}, we have $\lvert A\rvert =\gamma(H)=\tilde{q}(H)$ and there is a vertex in $A$ with degree $k(H)$.
	Noting that $\lvert V_0\rvert =\tilde{q}(H)-1$ holds and each edge between $V_0$ and $V'_1$ or $V'_2$ can be expanded into an odd cycle by using vertices of $V'_1$ and $V'_2$. If there is a copy of $S_{k(H)}(t)\subseteq G[V_1\cup V_2]$ and $V'_1$, $V'_2$ contain the neighbors of the center vertex of $S_{k(H)}$, there is a copy of $H(t)$ in $G$. 
Therefore, we may suppose $G[V_1\cup V_2]$ contains  no such   $S_{k(H)}(t)$. 
	
	Let $V'''_1\subseteq V''_1$, $V'''_2\subseteq V''_2$ be the vertex sets, which are not isolated vertices in $G[V''_1]$ and $G[V''_2]$ respectively. In the following part of the proof, denote by $G_1=G[V'''_1]$, $G_2=G[V'''_2]$, $G'=G_1\cup G_2$, $m=e(G')$. For a vertex $x\in V(G_i)$, denote by $E_{3-i}(x)=\{e\in E(G_{3-i})\,\vert V(e)\cap N_{G}(x)\neq \emptyset  \}$.
	
	If $e(G)<f(n,\,\tilde{q}(H))+{\rm {ex}}(\tilde{q}(H)-1,\,\tilde{\mathcal{B}}(H))+(k(H)-1)^2$ holds, the conclusion follows. Now suppose 
	\begin{equation}
		\begin{split}
			e(G)\geq& f(n,\,\tilde{q}(H))+{\rm {ex}}(\tilde{q}(H)-1,\,\tilde{\mathcal{B}}(H))+(k(H)-1)^2\\
			=&e(T_2(n-\tilde{q}(H)+1))+(\tilde{q}(H)-1)(n-\tilde{q}(H)+1)+{\rm {ex}}(\tilde{q}(H)-1,\,\tilde{\mathcal{B}}(H))+(k(H)-1)^2.
		\end{split}
	\end{equation}
	We  have the following claims.
	
	{\textbf{Claim 1.}} For every vertex $x\in V(G_i)$, we have  $d_{G_i}(x)+\alpha'(G_i-N_{G_i}[x])+\alpha'(G[E_{3-i}(x)]) \leq k(H)-1$ ($i=1,\,2$).

	Suppose to the contrary that there exists some $x\in V(G_i)$ such that $d_{G_i}(x)+\alpha'(G_i-N_{G_i}[x])+ \alpha'(G[E_{3-i}(x)])\geq k(H)$ holds. Without loss of generality, we may suppose $x\in V(G_1)$. Let $x_1,\,x_2,\cdots, x_s$ ($0\leq s\leq k(H)$) be $s$ neighbors of $x$ in $V''_1$; $y_{s+1}z_{s+1},\,\cdots, y_{u}z_{u}$ ($s\leq u\leq k(H)$) be a matching in $G_1-N_{G_1}[x]$; $y_{u+1}z_{u+1},\,\cdots, y_{k(H)}z_{k(H)}$ be a matching in $G[E_{2}(x)]$ where $x$ is adjacent to $y_{u+1},\,\cdots ,\,y_{k(H)}$. Since the number of vertices in $V'_1\cup V'_2$ is sufficiently large and each vertex of $V''_i$  is adjacent to at least $c_0\lvert V'_{3-i}\rvert $ vertices of $V'_{3-i}$ ($i=1,\,2$), we may find $k$ odd cycles intersecting in vertex $x$. When $1\leq j\leq s$, let $C_{t_j}=xP_{t_j-2}x_jx$. $P_{t_j-2}$ is a path between $V'_1$ and $V'_2$ and the endpoints of the path are in $V'_2$.
When $s+1\leq j\leq u$, let $C_{t_j}=xP_{t_j-4}z_jy_jw_jx$,  $P_{t_j-4}$ is a path between $V'_1$ and $V'_2$ and the endpoints of the path are in $V'_2$, $w_j$ is the vertex in $V'_2$.
When $u+1\leq j\leq k(H)$, let $C_{t_j}=xP_{t_j-3}z_jy_jx$,
 $P_{t_j-3}$ is a path between $V'_1$ and $V'_2$ and one of the endpoints of the path is in $V'_1$, the other endpoint of the path is in $V'_2$. For any $j\in [1,k(H)]$, $t_j\geq t$ is odd, the vertices of the paths in $C_{t_j}$ are different, $w_j$ is not in any paths in these cycles. Then there is a copy of $S_{k(H)}(t)\subseteq G[V_1\cup V_2]$ and $V'_1$ or $V'_2$ contains the neighbors of the center vertex of $S_{k(H)}$. This implies that there is a copy of $H(t)\subseteq G$,  a contradiction.
	%	\begin{equation*}
	%		\begin{split}
	%		C_{t_j}=&xv^j_{2,1}v^j_{1,2}\cdots v^j_{2,t_j-2}x_jx, \;\, 1\leq j\leq s;\\
	%		C_{t_j}=&xv^j_{2,1}v^j_{1,2}\cdots v^j_{2,t_j-4}y_jz_jv^j_{2,t_j-3}x, \;\, s+1\leq j\leq u;\\
	%	C_{t_j}=&xv^j_{2,1}v^j_{1,2}\cdots v^j_{1,t_j-3}z_{j}y_{j}x, \;\, u+1\leq j\leq k(H),
	%	\end{split}
	%\end{equation*}

	{\textbf{Claim 2.}} 
$\alpha'(G_1) +\alpha'(G_2)\leq k(H)-1$.
	
		Suppose to the contrary that $\alpha'(G_1) +\alpha'(G_2)\geq k(H)$. Let $\{x_1y_1,\,x_2y_2,\cdots,\,x_sy_s\}$ ($0\leq s\leq k(H)$) be a matching in $G_1$ and $\{x_{s+1}y_{s+1},\,\cdots,\,x_{k(H)}y_{k(H)}\}$  be a matching in $G_2$.  Since the number of vertices of $V'_1\cup V'_2$ is sufficiently large and each vertex of $V''_i$ ($i=1,\,2$) is adjacent to at least $c_0\lvert V'_{3-i}\rvert $ vertices of $V'_{3-i}$, we may find a vertex $v_0\in V'_1$ such that $v_0\in \cap _{i=s+1}^{k(H)} N_{V'_1}(x_i)$ holds and $k(H)$ odd cycles intersect in exactly one vertex $v_0$. When $1\leq j\leq s$, let $C_{t_j}=v_0P_{t_j-4}x_jy_jw_jv_0$, $P_{t_j-4}$ is a path between $V'_1$ and $V'_2$ and the endpoints of the path are in $V'_2$, $w_j$ is the vertex in $V'_2$.	When $s+1\leq j\leq k(H)$, let $C_{t_j}=v_0x_jy_jP_{t_j-3}v_0$, $P_{t_j-3}$ is a path between $V'_1$ and $V'_2$ and one of the endpoints of the path is in $V'_1$, the other endpoint of the path is in $V'_2$. For any $j\in [1,k(H)]$, $t_j\geq t$ is odd, the vertices of the paths in $C_{t_j}$ are different, $w_j$ is not in any paths in these cycles. Then there is a copy of $S_{k(H)}(t)\subseteq G[V_1\cup V_2]$ and $V'_1$ or $V'_2$ contains the neighbors of the center vertex of $S_{k(H)}$.
		This implies that there is a copy of $H(t)\subseteq G$, which is a contradiction. The result follows.
	
		%\begin{equation*}
		%	\begin{split}
		%		C_{t_j}=&v_0v^j_{2,1}v^j_{1,2}\cdots v^j_{2,t_j-4}x_jy_jv^j_{2,t_j-3}v_0, \quad 1\leq j\leq s;\\
			%	C_{t_j}=&v_0v^j_{2,1}v^j_{1,2}\cdots v^j_{1,t_j-3}x_jy_jv_0, \quad s+1\leq j\leq k(H),
		%	\end{split}
	%	\end{equation*}

	{\textbf{Claim 3.}} max $\{\Delta(G_1),\, \Delta(G_2)\}=k(H)-1$.
	
		By Claim 1 we have $\Delta(G_i)\leq k(H)-1$ ($i=1,\,2$). If max $\{\Delta(G_1),\, \Delta(G_2)\}\leq k(H)-2$, then 
		\begin{equation*}
			\begin{split}
				m=&e(G_1)+e(G_2)\\
				\leq& \varphi (\alpha'(G_1),\,k(H)-2)+\varphi(\alpha'(G_2),\,k(H)-2)\\
				\leq&\varphi(\alpha'(G_1)+\alpha'(G_2),\,k(H)-2)\\
				\leq &\varphi(k(H)-1,\,k(H)-2).
			\end{split}
		\end{equation*}
		By the construction of $G$, we deduce  $e(G)\leq f(n,\,\tilde{q}(H))+{\rm {ex}}(\tilde{q}(H)-1,\,\tilde{\mathcal{B}}(H)) +m$. Combining with (3.1), we have $m\geq (k(H)-1)^2$.
		If $k(H)=2$, then we have $m\leq \varphi (1,0)=0$, a contradiction. If $k(H)$ is odd, then we have $m\leq \varphi (k(H)-1,k(H)-2)<(k-1)^2$, a contradiction. If $k(H)$ is even and $k(H)\neq 4$, we have $m\leq \varphi(k(H)-1,\,k(H)-2)<(k(H)-1)^2$, a contradiction.
		
		If $k(H)=4$ holds, then we have $m\leq \varphi(3,\,2)=(k(H)-1)^2=9$. By (3.1), we have $m=(k(H)-1)^2= 9$. Then $G'$ is a graph with $e(G')=9$, $\Delta (G')=2$ and $\alpha'(G')=3$. Since $\Delta (G')=2$ and $G'$ has no isolated vertex, $\nu (G') \geq e(G')=9$, the equality holds if and only if $G'$ is 2-regular, and $c(G')\leq \alpha'(G')= 3$. On the other hand by Lemma \ref{j}, we obtain
		\begin{equation*}
			3=\alpha'(G')\geq \frac{\nu (G') -c(G')}{2}\geq \frac{\nu (G') -3}{2}.
		\end{equation*} 
		Hence $\nu (G')\leq 9$ and then  $\nu (G')=9$ holds. Therefore $G'$ consists of three vertex-disjoint triangles. As $m=(k(H)-1)^2$, then we have $e(G)=f(n,\,\tilde{q}(H))+{\rm {ex}}(\tilde{q}(H)-1,\,\tilde{\mathcal{B}}(H))+(k(H)-1)^2$. Therefore, each vertex of $V''_1$ is adjacent to each vertex of $V''_2$. Then for any vertex $x\in V(G')$, we have $d_{G_i}(x)+\alpha'(G_i-N_{G_i}[x])+\alpha'(G[E_{3-i}(x)]) =4=k(H)$ ($i=1,\;2$), a contradiction to Claim 1. Therefore, we have max $\{\Delta(G_1),\,\Delta(G_2)\}=k(H)-1.$

		{\textbf{Claim 4.}} 
		$e(G_1)\cdot e(G_2)=0$.
		
		First we have 
			\begin{equation*}
				\begin{split}
			m=&e(G_1)+e(G_2)\\
			\leq& \varphi(\alpha'(G_1),\,k(H)-1)+\varphi(\alpha'(G_2),\,k(H)-1)\\
			\leq &\varphi(\alpha'(G_1)+\alpha'(G_2),\,k(H)-1)\\
			\leq& \varphi(k(H)-1,\,k(H)-1)\\
			\leq& k(H)(k(H)-1).
			\end{split}
			\end{equation*}
		From Claim 3 we may suppose $\Delta(G_1)=k(H)-1$, and $x$ is in $V'''_1$ with $d_{G_1}(x)=k(H)-1$.  If $e(G_2)\geq 1$, then $\alpha'(G_2)\geq 1$. By Claim 2, $\alpha'(G_1)\leq k(H)-1-\alpha'(G_2)\leq k(H)-2$. By Claim 1, we obtain $\alpha'(G[E_2(x)])=0$ which implies $V'''_2\cap N_G(x)=\emptyset$. Hence, for every $v\in V'''_2$, $v$ is not adjacent to $x$.
		Let $n'=n-\tilde{q}(H)+1$. So $$e(V_1,\,V_2)\leq \lvert V_1\rvert \lvert V_2\rvert -\lvert V'''_2\rvert \leq e(T_2(n'))-\lvert V'''_2\rvert.$$ Thus we have 
		\begin{equation*}
			e(T_2(n'))+(k(H)-1)^2\leq e(G[V_1\cup V_2])\leq e(T_2(n'))-\lvert V'''_2\rvert+m.
		\end{equation*}
	Therefore, $\lvert V'''_2\rvert\leq m-(k(H)-1)^2$, and 
	\begin{equation*}
		\begin{split}
			m\leq &\varphi(\alpha'(G_1),\,\Delta (G_1))+\varphi(\alpha'(G_2),\,\Delta (G_2))\\
			\leq &\alpha'(G_1)(\Delta (G_1) +1)+\alpha'(G_2)(\Delta (G_2) +1)\\
			\leq& k(H)\alpha'(G_1)+(k(H)-1-\alpha'(G_1))\lvert V'''_2\rvert \\
			=&\alpha'(G_1)(k(H)-\lvert V'''_2\rvert )+(k(H)-1)\lvert V'''_2\rvert \\
			\leq &(k(H)-2)(k(H)-\lvert V'''_2\rvert)+(k(H)-1)\lvert V'''_2\rvert\\
			=&(k(H)-1)^2+\lvert V'''_2\rvert-1\\
			\leq& (k(H)-1)^2+m-(k(H)-1)^2-1\\
			=&m-1.
		\end{split}
	\end{equation*}
This contradiction shows $e(G_2)=0$. Therefore, we have proved $e(G_1)\cdot e(G_2)=0$.

	By Claim 1 and Claim 4,  for any vertex $x\in V_i$ we know that $d_{G_i}(x)+\alpha'(G_i-N_{G_i}[x])+\alpha'(G[E_{3-i}(x)]) \leq k(H)-1$ and $e(G_1)\cdot e(G_2)=0$ hold.  Then we have $d_{G'}(x)+\alpha'(G'-N_{G'}[x])\leq k(H)-1$. Applying Lemma \ref{k}, we deduce $e(G')\leq (k(H)-1)^2$. The equality holds if and only if $G'=K_{k(H)-1,\,k(H)-1}$. Therefore, we have $e(G)= f(n,\,\tilde{q}(H))+{\rm {ex}}(\tilde{q}(H)-1,\,\tilde{\mathcal{B}}(H))+(k(H)-1)^2$ and  $G\in \mathcal{F}(n,\,\tilde{q}(H),\,k(H)-1,\,\tilde{\mathcal{B}}(H))$. The proof is complete.
\end{proof}

~\

We mainly use the so-called progressive induction to prove the upper bound of Theorem \ref{d} and this technique is borrowed from \cite{ref22}.

\begin{theorem}[{Simonovites~\cite[]{ref68}}]\label{h}
	Let $\mathfrak{U}=\bigcup_{i=1}^{\infty} \mathfrak{U}_{i}$ be a set of given elements, such that $\mathfrak{U}_{i}$ are disjoint finite subsets of  $\mathfrak{U}$. Let $P$ be a condition or property defined on $\mathfrak{U}$ which means the elements of $\mathfrak{U}$ may satisfy or not satisfy $P$. Let $\phi(x)$ be a function defined on $\mathfrak{U}$ such that $\phi(x)$ is a non-negative integer and 
	\begin{enumerate}[(i)]
		\item if $x$ satisfies $P$, then $\phi(x)=0$;
		\item there is an $n_0$ such that if $n>n_0$ and $x\in \mathfrak{U}_n$ then either $x$ satisfies $P$ or there exist an $n'$ and an $x'$ such that 
		\begin{equation*}
			\frac{n}{2}<n'<n,\,x'\in \mathfrak{U}_{n'}\;\,and\;\, \phi(x)<\phi(x').
		\end{equation*}
		Then there exists an $n_0$ such that  if $n>n_0$, every $x\in \mathfrak{U}_n$ satisfies $P$.
	\end{enumerate}
\end{theorem}
\
\newline
\textbf{Proof of Theorem \ref{d} } 
%\begin{equation*}
%	f(n,\,\tilde{q}(H)) +{\rm {ex}}(\tilde{q}(H)-1,\,\tilde{\mathcal{B}}(H))\leq {\rm {ex}}(n,\,H(t))\leq f(n,\,\tilde{q}(H))+{\rm {ex}}(\tilde{q}(H)-1,\,\tilde{\mathcal{B}}(H))+(k(H)-1)^2.
%\end{equation*}
Let $G_n$ be an extremal graph for $H(t)$ of order $n$, $F_n\in \mathcal{F}(n,\,\tilde{q}(H),\,k(H)-1,\,\tilde{\mathcal{B}}(H))$, then $e(F_n)=f(n,\,\tilde{q}(H))+{\rm {ex}}(\tilde{q}(H)-1,\,\tilde{\mathcal{B}}(H))+(k(H)-1)^2$. Let $\mathfrak{U}_n$ be the set of extremal graphs for $H(t)$ of order  $n$, $P$ be the property defined on $\mathfrak{U}$ satisfying that $e(G_n)\leq  e(F_n)$ and the equality holds if and only if $G_n\in \mathcal{F}(n,\,\tilde{q}(H),\,k(H)-1,\,\tilde{\mathcal{B}}(H))$. Define  $\phi (G_n)=max\{e(G_n)- e(F_n),\;0\}$. If $G_n$ satisfies $P$, then $\phi (G_n)=0$, which implies the condition (i) in Theorem \ref{h} is satisfied. 
 
 In the following part we may prove either $G_n$ satisfies $P$ or there exist an $n'$ and a $G_{n'}$ such that
\begin{equation*}
	\frac{n}{2}<n'<n,\;G_{n'}\in \mathfrak{U}_{n'}\;\,and \;\,\phi(G_{n})<\phi (G_{n'}).
\end{equation*}

By Theorem 1.2 and the fact $e(G_n)\geq f(n,\,\tilde{q}(H))+{\rm {ex}}(\tilde{q}(H)-1,\,\tilde{\mathcal{B}}(H))\geq \frac{n^2}{4}$, there is an $n_1$ such that if $n>n_1$, $G_n$ contains $T_{2}(2n_2)$ ($n_2$ is sufficiently large) as a subgraph. By Lemma \ref{2.3}, we have $e(H)P_2\in \mathcal{M}_2(H(t))$. In $G_n$ each class of $T_2(2n_2)$ contains no copy of $e(H)P_2$. Otherwise it follows from the definition of the 2-decomposition family that $G_n$ contains a copy of $H(t)$, a contradiction. Hence, there is an induced subgraph $T_2(2n_3)$ ($n_3$ is also sufficiently large) of $G_n$ by deleting $2e(H)$ vertices of each class of $T_2(2n_2)$.

Let $c$ be a sufficiently small constant and $T_0=T_2(2n_3)$, $X=V(G_n)\backslash V(T_0)$. We  pick vertices $x_t\in X$ and graphs $T_t$ recursively: $x_t$ is the vertex which has at least $c^tn_3$ neighbors in each class of $T_{t-1}$, and $T_t=T_2(2c^tn_3)$ is the subgraph of $T_{t-1}$ induced by the neighbors of $x_t$. $B^t_1$ and $B^t_2$ are the vertex sets of  two classes of $T_t$. The progress stops after at most $\tilde{q}(H)-1$ steps. If $t\geq \tilde{q}(H)$,  $G_n[\{x_1,\,x_2,\cdots ,x_{\tilde{q}(H)}\}\cup B^{\tilde{q}(H)}_1]$  contains a copy of $M$, then $G_n$  contains a copy of $H(t)$. Therefore, we may suppose the progress ends at $x_s$ and $T_s$ where $s\leq \tilde{q}(H)-1$. Denote by  $Y=\{x_1,\,\cdots, x_s\}$.

Next we divide $V\left(G_n\right)\backslash \left(V\left(T_s\right)\cup Y\right)$. If $x\in V\left(G_n\right)\backslash \left(V\left(T_s\right)\cup Y\right)$ is adjacent to less than $c^{s+1}n_3$ vertices of $B^s_i$ and is adjacent to at least $(1-\sqrt{c})c^sn_3$ vertices of $B^s_{3-i}$, then we put $x$ in $C_i$ ($i=1,\,2$). If $x\in V\left(G_n\right)\backslash \left(V\left(T_s\right)\cup Y\right)$ is adjacent to less than $c^{s+1}n_3$ vertices of $B^s_i$ and is adjacent to less than $(1-\sqrt{c})c^sn_3$ vertices of $B^s_{3-i}$ for some $i\in \{1,2\}$, then we put $x$ in $D$. Then $V(G_n)\backslash \left(V(T_s)\cup Y\right)=C_1\cup C_2\cup D$ holds. 

The number of independent edges in $G_n[B^s_i\cup C_i]$ is less than $e(H)$.  Otherwise, if $e(H)P_2\subseteq G_n[B^s_i\cup C_i]$, $G_n$ contains a copy of $H(t)$.
 Consider the edges joining $B^s_{i}$ and $C_i$ and select a maximal set of independent edges, say $y_1z_1,\,\cdots,y_mz_m$ with $y_j\in B^s_i$, $z_j\in C_i$ and $1\leq j\leq m$, $1\leq m< \ell $, where $\ell =e(H)$. The number of vertices of $B^s_i$ joining to at least one of $z_1,\,z_2,\cdots,z_m$
is less than $c^{s+1}\ell n_3$ and the remaining vertices of $B^s_i$ are not adjacent to any vertex of $C_i$. Therefore there are at least $(1-c\ell )c^{s}n_3$ vertices of $B^{s}_i$ which are not adjacent to any vertices of $C_i$. We may move these $c^{s+1}\ell n_3$ vertices of $B^s_i$ to $C_i$ to obtain $B'_i$ and $C'_i$ such that  $B'_i\subseteq B^s_i$, $C_i\subseteq C'_i$ and there are no edges between $B'_i$ and $C'_i$.

In conclusion, the vertices of $G_n$ can be partitioned into $V(T'_{s})$, $C'_1$, $C'_2$, $D$ and $Y$, where  $T'_{s}=T_2(2n_4)$ with classes $B'_1$ and $B'_2$, $n_4=c^sn_3-c^{s+1}\ell n_3$.
\begin{enumerate}[(i)]
	\item $\lvert Y\rvert =s$ and each $v\in Y$ is adjacent to each vertex of $T_2(2n_4)$.
	\item Each vertex of  $ C'_i$ is adjacent to at least $(1-\sqrt{c}-c\ell )c^sn_3$ vertices of $B'_{3-i}$ and is not  adjacent to any vertex of $B'_i$ ($i=1,\,2$).
	\item Each vertex of  $ D$ is adjacent to less than $c^{s+1}n_3$ vertices of $B'_i$ and is adjacent to less than $(1-\sqrt{c})c^sn_3$ vertices of $B'_{3-i}$ for some $i\in \{1,2\}$.
\end{enumerate}

Denote by $\widehat{G}=G_n-V(T'_{s})$. Since $\widehat{G}$ does not contain a copy of $H(t)$, we have $e(\widehat{G})\leq e(G_{n-2n_4})$.
There is a $T'_{s}$ contained in $F_n$.  Denote by $\widehat{F}=F_n-V(T'_{s})$. Then
\begin{equation*}
	\begin{split}
	e(G_n)-e(F_n)=&e(T_{s}')+e\left(V(\widehat{G}),\,V(T'_{s})\right)+e(\widehat{G})-\left[e(T_{s}')+e\left(V(\widehat{F}),\,V(T'_{s})\right)+e(\widehat{F})\right]\\
	=&e(\widehat{G})-e(\widehat{F})+e\left(V(\widehat{G}),\,V(T'_{s})\right)-e\left(V(\widehat{F}),\,V(T'_{s})\right)\\
	\leq &e(G_{n-2n_4})-e(F_{n-2n_4})+e\left(V(\widehat{G}),\,V(T'_{s})\right)-e\left(V(\widehat{F}),\,V(T'_{s})\right).
	\end{split}
\end{equation*}
Then we have $\phi (G_n)\leq \phi (G_{n-2n_4}) +e\left(V(\widehat{G}),\,V(T'_{s})\right)-e\left(V(\widehat{F}),\,V(T'_{s})\right)$.

On the other hand 
\begin{equation*}
	\begin{split}
	&e\left(V(\widehat{G}),\,V(T'_{s})\right)-e\left(V(\widehat{F}),\,V(T'_{s})\right)\\
	\leq &2sn_4+(n-s-2n_4-\lvert D\rvert)n_4+\lvert D\rvert [c^{s+1}n_3+(1-\sqrt{c})c^sn_3]\\
	&-[(\tilde{q}(H)-1)2n_4+n_4(n-\tilde{q}(H)+1-2n_4)]\\
	=& [2s+n-s-\lvert D\rvert -2(\tilde{q}(H)-1)-(n-\tilde{q}(H)+1)]n_4\\
	&+[\lvert D\rvert c^{s+1}+\lvert D\rvert (1-\sqrt{c}) c^{s}]n_3\\
	=& (s-\tilde{q}(H)+1)n_4+(c(\ell +1)-\sqrt{c})c^sn_3\lvert D\rvert\\
	\leq& 0.
	\end{split}
\end{equation*}
	
	If $e\left(V(\widehat{G}),\,V(T'_{s})\right)-e\left(V(\widehat{F}),\,V(T'_{s})\right) <0$, then $\phi (G_n)<\phi (G_{n-2n_4})$ holds. Since $n-2n_4 >\frac{n}{2}$, the condition (ii) in Theorem \ref{h} is satisfied. 
	
	If $e\left(V(\widehat{G}),\,V(T'_{s})\right)-e\left(V(\widehat{F}),\,V(T'_{s})\right) =0$, then $s=\tilde{q}(H)-1$, $\lvert D\rvert =0$, each vertex of $C'_i$ is adjacent to at least $(1-\sqrt{c}-c\ell )c^sn_3$ (not less than $c_0n_4$, $c_0\in (\frac{1}{2},1]$ is a constant) vertices of $B'_{3-i}$ ($i=1,\,2$). By Theorem \ref{m}, we have  $e(G_n)\leq  f(n,\,\tilde{q}(H))+{\rm {ex}}(\tilde{q}(H)-1,\,\tilde{\mathcal{B}}(H)) +(k(H)-1)^2$, and the equality holds if and only if $G_n\in \mathcal{F}(n,\,\tilde{q}(H),\,k(H)-1,\,\tilde{\mathcal{B}}(H))$. 
	
	Since $n$ is sufficiently large, there exists an $n_0$ such that $n>n_0$. Hence  ${\rm {ex}}(n,\,H(t))\leq  f(n,\,\tilde{q}(H))+{\rm {ex}}(\tilde{q}(H)-1,\,\tilde{\mathcal{B}}(H)) +(k(H)-1)^2$, and the equality holds if and only if the extremal graphs for $H(t)$ of order $n$ are in $\mathcal{F}(n,\,\tilde{q
	}(H),\,k(H)-1,\,\tilde{\mathcal{B}}(H))$.
	The proof is complete.

	~\\
	{\textbf{Remark}}
	
	This paper determines the range of Tur\'{a}n numbers for odd-ballooning of general bipartite graphs obtained from replacing each edge by an odd cycle of order $t$ where $t\geq 5$ is an odd integer. Given an integer $p$,
	the edge blow-up of a graph $H$, denoted by $H^{p+1}$, is the graph obtained from replacing each edge in $H$ by a clique of order $p+1$, and the new vertices of the cliques are all distinct. Yuan in \cite{ref22} determined the range of Tur\'{a}n numbers for edge blow-up of all bipartite graphs when $p\geq 3$.  The Tur\'{a}n numbers for $H(3)$ has been determined when $H$ is a star, a path or an even cycle. While the Tur\'{a}n numbers for $H(3)$ when $H$ is a general bipartite graph are unclear.
	
	%The  for $H(3)$ has been determined in \cite{ref18,ref202} when $H$ is a star or a path or an even cycle. When $H$ is a tree, the Tur\'{a}n numbers of a part of the graphs in $H(3)$ has been studied in \cite{ref221}.

~\\
{\textbf{Declaration}}

The authors have declared that no competing interest exists.

\end{document}